\documentclass[a4paper,reqno,oneside,11pt]{amsart}
\usepackage{fullpage}

\usepackage{amsmath}
\usepackage{amsfonts}
\usepackage{amssymb}
\usepackage{verbatim,enumitem,graphicx}
\usepackage[colorlinks]{hyperref}
\usepackage{dsfont}
\usepackage{tikz}

\usepackage[utf8]{inputenc}

\usepackage{enumitem}
\usepackage{MnSymbol}

\setenumerate{leftmargin=*} % ,labelindent=\parindent}

\newtheorem{theorem}{Theorem}[section]
\newtheorem{lemma}[theorem]{Lemma}
\newtheorem{proposition}[theorem]{Proposition}
\newtheorem{corollary}[theorem]{Corollary}

\newtheorem{definition}[theorem]{Definition}

\newtheorem{claim}[theorem]{Claim}

\newcommand{\oldqed}{}
\def\endofClaim{\hfill\scalebox{.6}{$\Box$}}

\newcommand{\cC}{\mathcal{C}}

\newcommand{\cF}{\mathcal{F}}

\newcommand{\cK}{\mathcal{K}}
\newcommand{\cI}{\mathcal{I}}

\newcommand{\PP}{\mathbb{P}}
\newcommand{\EE}{\mathbb{E}}

\newcommand{\eps}{\varepsilon}

\renewcommand{\subset}{\subseteq}

\def\itm#1{\rm ({#1})}
\def\itmit#1{\itm{\it #1\,}}
\def\rom{\itmit{\roman{*}}}
\def\abc{\itmit{\alph{*}}}

\begin{document}
	
\title{Cycle factors in randomly perturbed graphs$^\ast$}

\author[J.~B\"{o}ttcher]{Julia B\"{o}ttcher$^\dag$}
\author[O.~Parczyk]{Olaf Parczyk$^{\dag,\S}$}
\author[A.~Sgueglia]{Amedeo Sgueglia$^\dag$}
\author[J.~Skokan]{Jozef Skokan$^\dag$$^\ddag$}

\thanks{$^\dag$ Department of Mathematics London School of Economics London, WC2A 2AE, UK \\
Email: \href{mailto:a.sgueglia@lse.ac.uk}{\{\texttt{j.boettcher|o.parczyk|a.sgueglia|j.skokan}\}\texttt{@lse.ac.uk}}}
\thanks{$^\ddag$ Department of Mathematics, University of Illinois at Urbana-Champaign,  Urbana, IL 61801, USA}
\thanks{$^\S$ OP was supported by the Deutsche Forschungsgemeinschaft (DFG, Grant PA 3513/1-1).}
\thanks{$^\ast$ An extended abstract of this work will appear in the proceedings of LAGOS $2021$.}

\begin{abstract} 
  We study the problem of finding pairwise vertex-disjoint copies of the
  $\ell$-vertex cycle $C_\ell$ in the randomly perturbed graph model, which is
  the union of a deterministic $n$-vertex graph $G$  and the
  binomial random graph $G(n,p)$. For $\ell \ge 3$ we prove that asymptotically
  almost surely $G \cup G(n,p)$ contains $\min \{\delta(G), \lfloor n/\ell
  \rfloor \}$ pairwise vertex-disjoint cycles $C_\ell$, provided $p \ge C \log
  n/n$ for~$C$ sufficiently large. Moreover, when $\delta(G) \ge\alpha n$ with
  $0<\alpha \le 1/\ell$ and~$G$ and is not `close' to the complete bipartite
  graph $K_{\alpha n,(1-\alpha) n}$, then $p \ge C/n$ suffices to get the same
  conclusion. This provides a stability
  version of our result.
  In particular, we conclude that
  $p \ge C/n$ suffices when $\alpha>1/\ell$ for finding $\lfloor
  n/\ell \rfloor$ cycles $C_\ell$.

  Our results are asymptotically optimal. They can be seen as an interpolation
  between the Johansson--Kahn--Vu Theorem for $C_\ell$-factors and the
  resolution of the El-Zahar Conjecture for $C_\ell$-factors by Abbasi.
\end{abstract}

\maketitle

\section{Introduction and results}

Given a graph $H$, deciding whether a graph~$F$ has an \emph{$H$-factor}, i.e.~the union of
$\lfloor v(F)/v(H) \rfloor$ pairwise vertex-disjoint copies of $H$, is
computationally hard~\cite{KirkpatrickHell} already when $H$ is a triangle.
Consequently, it is valuable to determine natural sufficient conditions on~$F$
which guarantee an $H$-factor. Two natural and prominent conditions of this type
concern minimum degree conditions on the one hand, and edge densities in the
setting of random graphs on the other hand. In this extended abstract we
concentrate on the case that~$H$ is the $\ell$-vertex cycle $C_\ell$ with
$\ell \ge 2$; in the degenerate case $\ell=2$ the cycle $C_\ell$ is a single
edge and we obtain a perfect matching. We shall first summarise what is known
in these two different settings, and then consider a well-studied combination of
both, the so called randomly perturbed graph model.

Let us first consider the case when $F$ is the binomial random graph $G(n,p)$,
which is a graph on~$n$ vertices in which each edge is chosen independently with
probability $p=p(n)$. In this case we are interested in the probability
\emph{threshold} $\hat{p}=\hat{p}(n,H)$ such that \emph{asymptotically almost
  surely}~(a.a.s.), that is, with probability tending to one as $n$ tends to
infinity, $G(n,p)$ contains an $H$-factor when $p =\omega (\hat{p})$, and
a.a.s.~$G(n,p)$ does not contain an $H$-factor when $p =o(\hat{p})$. For
$H=C_\ell$ with $\ell \ge 3$, the celebrated theorem of Johansson, Kahn, and
Vu~\cite{johansson2008factors} implies that the threshold is
$\hat{p}(n,C_\ell)=n^{-(\ell-1)/\ell} (\log n)^{1/\ell}$, where in the case of a perfect matching ($\ell=2$) the threshold $\log n/n$ has been known since the seminal work of
Erd\H{o}s and R\'enyi~\cite{erdHos1966existence}.

Turning to minimum degree conditions enforcing the existence of
$C_\ell$-factors, let $G_{\alpha}$ be any $n$-vertex graph of minimum degree at
least $\alpha n$ for $0\le\alpha\le 1$. By Dirac's Theorem~\cite{dirac1952},
$\alpha\ge 1/2$ suffices for guaranteeing a perfect matching. Corr\'adi and
Hajnal~\cite{corradi1963maximal}, on the other hand, showed that $\alpha\ge 2/3$
suffices for guaranteeing a $C_3$-factor. Abbasi~\cite{abbasi1998solution}
generalised this, confirming more generally a conjecture of El-Zahar, showing
that any graph $G$ with minimum degree $\delta(G) \ge \tfrac n\ell \cdot \lceil
\tfrac \ell2 \rceil$ contains a $C_\ell$-factor. Note that this implies that the
case of even~$\ell$ and that of odd~$\ell$ behave differently: for even $\ell$
we need $\delta(G)\ge n/2$, while for odd $\ell$ we need
$\delta(G)\ge\tfrac{\ell+1}{2\ell} n$. This is not surprising, as in general the
optimal minimum degree enforcing an $H$-factor depends on the chromatic number
(or some variant, called the critical chromatic number). See the survey by
K\"uhn and Osthus~\cite{kuhn2009embedding} for more details.

These results provide optimal minimum degree conditions. This can be easily
seen by taking the complete bipartite graph with partition classes of sizes $n/2-1$ and
$n/2+1$ for~$\ell$ even, and the
complete tripartite graph with classes of sizes $\frac{n}{\ell}-1$,
$\frac{\ell-1}{2 \ell}n+1$ and $\frac{\ell-1}{2 \ell}n$ for $\ell$~odd.

\smallskip

Bohman, Frieze, and Martin~\cite{bohman2003} combined these two settings,
introducing the \emph{randomly perturbed graph} $G_\alpha \cup G(n,p)$, which is
obtained by adding to a deterministic graph $G_\alpha$ on~$n$ vertices with
minimum degree at least $\alpha n$, a random graph graph $G(n,p)$ on the same
vertex set. This model can be motivated as follows. We can alternatively think
of $G(n,p)$ as a graph obtained by the following \emph{random process}: start
with the empty graph on~$n$ vertices and add random edges, one by one. Asking
how many random edges we need to add to guarantee a certain property then
corresponds to determining the threshold for this property. In fact, as
formulated, this process does not generate $G(n,p)$ but rather the uniform
random graph $G(n,M)$; but this is easy and standard to fix. It is then natural
to modify this process by starting, instead of the empty graph, with some other
deterministic $n$-vertex graph, for example a graph $G_\alpha$ with minimum
degree $\alpha n$. The question then is how this influences the number of random
edges needed to enforce the considered property. When $\alpha>0$ is small then
this can be seen as asking how much the threshold for the property is influenced
by the existence of low-degree vertices. When $p$ is small then, in analogy to
the \emph{smoothed analysis of algorithms} introduced Spielman and
Teng~\cite{spielman2004smoothed}, this question can be seen as asking how
`atypical' extremal graphs for the property are.

With this in mind, for a fixed $\alpha>0$ and $H$, we define the
\emph{perturbed threshold} for an $H$-factor as the $\hat{p}=\hat{p}(n,\alpha,H)$
such that:
\begin{enumerate}[label=\rom]
	\item when $p=\omega(\hat{p}(n,\alpha,H))$, for any $G_\alpha$ we have 
	\[ \lim_{n\to\infty}\PP\big(G_\alpha \cup G(n,p) \text{ contains an $H$-factor}\big) =1, \]
	and
	\item when $p=o(\hat{p}(n,\alpha,H))$, there exists a $G_\alpha$ such that 
	\[ \lim_{n\to\infty}\PP\big(G_\alpha \cup G(n,p) \text{ contains an $H$-factor}\big) =0. \]
\end{enumerate}

Balogh, Treglown, and Wagner~\cite{balogh2019tilings} proved a lower bound on
$\hat{p}(n,\alpha,H)$ for any $H$, which is sharp for all $H$ provided
$\alpha$ is small enough. In our setting where $H=C_\ell$, their result
states that for any $\ell \ge 3$ and for any $\alpha >0$, there is a constant
$C=C(\alpha,\ell)$ such that $G_\alpha \cup G(n,p)$ with $p \ge C
n^{-(\ell-1)/\ell}$ a.a.s.\ contains a $C_\ell$-factor; their result gives the lower bound of $p \ge C/n$ 
also in the case $\ell=2$, which
was already proven in~\cite{bohman2003}. Compared to the threshold
$n^{-(\ell-1)/\ell} (\log n)^{1/\ell}$ in $G(n,p)$ alone this saves a
$\log$-factor. By taking $G_\alpha$ to be the complete bipartite graph
$K_{\alpha n,(1-\alpha)n}$, it is easy to see that this is optimal for $\alpha <
1/\ell$, so $\hat{p}(n,\alpha,C_\ell)=n^{-(\ell-1)/\ell}$ when $0 < \alpha
<1/\ell$.

The problem of determining the perturbed threshold for the remaining range of
$\alpha$ (that is, $\alpha \in [1/\ell,1/2)$ for $\ell$ even, and $\alpha \in
[\tfrac 1\ell, \tfrac {\ell+1}{2\ell})$ for $\ell$ odd) remained open. In fact,
this `intermediate regime' where $\alpha$ is not small but potentially far
from the extremal bound in the deterministic setting has so far only
infrequently been studied for randomly perturbed graphs. One
exception is the work by Han, Morris, and Treglown~\cite{han2019tilings}
concerning clique-factors and proving in particular that
$\hat{p}(n,\alpha,C_3)=n^{-1}$ for $\alpha \in (1/3,2/3)$. We recently filled in
the remaining gap $\alpha=1/3$ and proved $\hat{p}(n,1/3,C_3)=\log n/n$
in~\cite{triangle_paper}. In this extended abstract we generalise these results to larger $\ell$
and determine the perturbed threshold $\hat{p}(n,\alpha,C_\ell)$ in all open cases for $C_\ell$-factors
with $\ell \ge 3$.

\begin{theorem}
	\label{thm:zero}
	For any integer $\ell \ge 3$ and any
	$\alpha \ge 1/\ell$, there exists $C>0$ such that for any
  $n$-vertex graph $G_\alpha$ with minimum degree $\delta(G_{\alpha})=\alpha n$, the
  randomly perturbed graph $G \cup G(n,p)$ a.a.s.\ contains a $C_\ell$-factor
	\begin{enumerate}[label=\abc] 
	\item \label{thm:no-log}
    if $\alpha > 1/\ell$ and $p\ge C /n$, and also
   	\item \label{thm:log}
    if $\alpha=1/\ell$ and $p\ge C\log n/n$.
	\end{enumerate}
\end{theorem}

The bound on $p$ in~\ref{thm:log} is asymptotically optimal. To see this, take
$p \le \tfrac 12 \frac{\log n}{n}$ and $G=G_\alpha$ to be the graph $K_{n/\ell,n-n/\ell}$.
Let $A$ and $B$ be its partition classes with $|A|<|B|$, and observe this graph has minimum degree
$n/\ell$.
By an easy first-moment calculation, a.a.s.~there is at least a
polynomial number of vertices in $B$ that only have neighbours in $A$.
In particular, any cycle containing one of such vertices must contain
at least two vertices from $A$.
However, if a $C_\ell$-factor exists in $G \cup G(n,p)$, since $|A|=n/\ell$,
for each copy of $C_\ell$ that has at least two
vertices in $A$, there must be at least one copy of $C_\ell$ fully contained
in $B$, and thus with all edges from $G(n,p)$.
Again by an easy first-moment calculation, a.a.s.~there are at most
$O(\log^\ell n)$ copies of $C_\ell$ in $G(n,p)$ alone.
Therefore a.a.s.~a $C_\ell$-factor does not exist in $G \cup G(n,p)$
for $p \le \tfrac 12 \frac{\log n}{n}$.
Together with~\ref{thm:log} this implies that $\hat{p}(n,1/\ell,C_\ell)=\log
n/n$.

Now we turn to the optimality of~\ref{thm:no-log}. For even $\ell \ge 4$ and
$\alpha \in (1/\ell , 1/2)$, we consider $G=K_{\alpha n, (1-\alpha)n}$. It has
minimum degree $\alpha n$ and there can be at most $\frac{\alpha
  n}{\ell/2}<\frac{n}{\ell}$ copies of $C_\ell$ using only edges of $G$ and,
therefore, we need at least a linear number of random edges. Together
with~\ref{thm:no-log} this implies that when $\ell$ is even,
$\hat{p}(n,\alpha,C_\ell)=1/n$ for $\alpha \in (1/\ell , 1/2)$. For odd $\ell
\ge 3$, then the same graph shows optimality for $\alpha \in (1/\ell , 1/2)$, as
$G$ is bipartite and does not contain any odd cycle and again we need at least a
linear number of random edges. For $\alpha \in [\tfrac 12,
\tfrac{\ell+1}{2\ell})$, we consider the tripartite complete graph with one
class of size $\left( \alpha -\frac{\ell-1}{2\ell} \right) n$ and two classes of
sizes $\left( \frac{1}{2}-\frac{\alpha}{2}+\frac{\ell-1}{4\ell} \right) n$. This graph
has minimum degree $\alpha n$ and, as above, there are at most $(\alpha
-\frac{\ell-1}{2\ell}) n < \frac{n}{\ell}$ copies of $C_\ell$ using only edges of $G$
and we need a linear number of edges from $G(n,p)$. Together
with~\ref{thm:no-log} this implies that when $\ell$ is odd,
$\hat{p}(n,\alpha,C_\ell)=1/n$ for $\alpha \in (\tfrac 12,
\tfrac{\ell+1}{2\ell})$.
Table~\ref{fig:cycle-factor} summarises the resulting perturbed
thresholds for cycle factors.

\begin{table}[htbp]
	\begin{center}
	\begin{tabular}{||c||c| c c c c c||} 
		\hline
		Even $\ell$ & $\alpha$ & $\alpha=0$ & $0<\alpha<1/\ell$ & $\alpha=1/\ell$ & $1/\ell<\alpha<1/2$ & $1/2 \le \alpha$\\
		\hline\hline
		Odd $\ell$ & $\alpha$ & $\alpha=0$ & $0<\alpha<1/\ell$ & $\alpha=1/\ell$ & $1/\ell<\alpha<\tfrac{\ell+1}{2\ell}$ & $\tfrac{\ell+1}{2\ell} \le \alpha$\\
		\hline\hline
		& $\hat p$ & $n^{-(\ell-1)/\ell} (\log n)^{1/\ell}$ & $n^{-(\ell-1)/\ell}$ & $n^{-1} \log n$ & $n^{-1}$ & $0$ \\ 
		\hline
	\end{tabular}	
	\end{center}
	\caption{The perturbed threshold $\hat p=\hat p(n,\alpha,C_\ell)$ for $C_\ell$-factor in $G_\alpha \cup G(n,p)$, where $\delta(G_\alpha) \ge \alpha n$.}
	\label{fig:cycle-factor}
\end{table}

It is, further, natural to ask how this behaviour changes when instead of a
$C_\ell$-factor we are interested in covering only a smaller percentage of the
vertices with vertex disjoint $C_\ell$-copies. To this end, we can prove that we
can always find $\delta(G)$ pairwise vertex disjoint copies of $C_\ell$ in $G\cup
G(n,p)$ when $p \ge C \log n/n$.

\begin{theorem}
	\label{thm:main}
    For any integer $\ell \ge 3$, there exists a $C>0$ such that for any
    $n$-vertex graph $G$ we can a.a.s.~find $\min\{ \delta(G), \lfloor n/\ell
    \rfloor \}$ pairwise disjoint copies of $C_\ell$ in $G \cup G(n,p)$, provided that $p\ge C \log n /n$.
\end{theorem}

For $\ell=3$ this is a perturbed version of a result of
Dirac~\cite{dirac1963maximal}, that states that an $n$-vertex graph $G$ with
$\tfrac 12 n \le \delta(G) \le \tfrac 23 n$ contains at least $2 \delta(G)-n$
pairwise vertex-disjoint triangles. For $\ell \ge 4$ an approximate version of
this result for longer cycles follows from a more general result of
Koml{\'o}s~\cite{komlos2000tiling}. More precisely, for any $\eps>0$ and large
enough $n$, he showed that when $\ell$ is odd (respectively even) there are at least $2
\delta(G)-(1+\eps)n$ (respectively $\tfrac 2\ell \delta(G)-\eps n$) pairwise
vertex-disjoint copies of $C_{\ell}$ in any $n$-vertex graph $G$ with $\tfrac 12
n \le \delta(G) \le \tfrac{\ell+1}{2 \ell}n$ (respectively $\delta(G) \le
\frac{1}{2}n$).

Moreover, we establish a stability version of Theorem~\ref{thm:main}:
when the graph $G$ has minimum degree linear in $n$ and $G$ is not `close' to
$K_{m,n-m}$ with $m=\min \{\delta(G),n/\ell \}$, then a.a.s $G \cup G(n,p)$
contains $m$ pairwise vertex-disjoint copies of $C_\ell$ already at probability
$p \ge C/n$. To formalise this we introduce the following notion of stability,
where, for numbers $a$, $b$, $c$, we write $a \in b \pm c$ for $b-c \le a \le b+c$.

\begin{definition}
	\label{def:stability}
	For $0<\beta < \alpha < 1/2$ we say that an $n$-vertex graph $G$ is
  \emph{$(\alpha,\beta)$-stable} if there exists a partition of $V(G)$ into two
  sets $A$ and $B$ of size $|A|=(\alpha \pm \beta) n$ and $|B|=(1-\alpha \pm
  \beta) n$ such that the minimum degree of the bipartite subgraph $G[A,B]$ of
  $G$ induced by $A$ and $B$ is at least $\alpha n/4$, all but $\beta n$
  vertices from $A$ have degree at least $|B|-\beta n$ into $B$, all but
  $\beta n$ vertices from $B$ have degree at least $|A|-\beta n$ into $A$, and $G[B]$ contains at most $\beta n^2$ edges.
\end{definition}

Roughly speaking, the stability condition with $\alpha=1/\ell$ says that the
size of $B$ is roughly $(\ell-1)$-times the size of $A$, there is a minimum
degree condition between $A$ and $B$, in each part all but few vertices see
most of the other part, and $B$ is almost independent.
Moreover, for $0 < \alpha \le 1/\ell$ and $m=\alpha n$ an
integer, $K_{m,n-m}$ is $(\alpha,0)$-stable.
Using regularity method, we can prove the following stability result.

\begin{theorem}[Stability Theorem]
	\label{thm:non-extremal}
	Fix an integer $\ell \ge 3$. For $0 < \beta < 1/(4 \ell)$ there exist $\gamma>0$ and
  $C>0$ such that for any $\alpha$ with $4 \beta \le \alpha \le 1/\ell$ the
  following holds. Let $G$ be an $n$-vertex graph with minimum degree $\delta(G)
  \ge \left(\alpha - \gamma \right) n$ that is not $(\alpha,\beta)$-stable. With
  $p\ge C /n$ a.a.s.~the perturbed graph $G \cup G(n,p)$ contains $ \min \{
  \alpha n, \lfloor n/\ell \rfloor \} $ pairwise vertex-disjoint disjoint copies
  of~$C_\ell$.
\end{theorem}

Note that this immediately implies~\ref{thm:no-log} of Theorem~\ref{thm:zero}.
Indeed, given an integer $\ell \ge 3$ and $\alpha > 1/\ell$, there is a small 
enough $\beta$ such that any $n$-vertex graph $G$ with minimum degree at least $\alpha n$
is not $(1/\ell,\beta)$-stable.
We can then apply Theorem~\ref{thm:non-extremal}
on input $\beta$ (and by taking $\alpha=1/\ell$), which always gives
a $C_\ell$-factor.
Moreover, when we restrict to graphs $G$ that are not $(\alpha,\beta)$-stable, then the constant $C$ only depends on $\ell$ and $\beta$, but is independent of $\alpha$.
To deal with $(\alpha,\beta)$-stable graphs for a small enough $\beta>0$,
we need the $\log n$-factor and we prove the following.

\begin{theorem}[Extremal Theorem]
	\label{thm:extremal}
	Fix an integer $\ell \ge 3$. For $0 < \alpha_0 \le 1/\ell$ there exist
  $\beta,\gamma>0$ and $C>0$ such that for any $\alpha$ with $\alpha_0 \le
  \alpha \le 1/\ell$ the following holds. Let $G$ be an $n$-vertex graph with
  minimum degree $\delta(G) \ge \left( \alpha -\gamma \right) n$ that is
  $(\alpha,\beta)$-stable. With $p\ge C \log n /n$ a.a.s.~the perturbed graph $G
  \cup G(n,p)$ contains $\min\{ \delta(G), \lfloor \alpha n \rfloor \}$ pairwise
  vertex-disjoint copies of $C_\ell$.
\end{theorem}

Together with Theorem~\ref{thm:non-extremal} this implies~\ref{thm:log} of
Theorem~\ref{thm:zero}. When the minimum degree is smaller, we prove the
following result.

\begin{theorem}[Sublinear Theorem]
	\label{thm:sublinear}
	Fix an integer $\ell \ge 3$. There exists a $C>0$ such that the following
  holds for any $1 \le m \le \frac{n}{64 \ell^2}$ and any $n$-vertex graph $G$
  of minimum degree $\delta(G) \ge m$. With $p \ge C \log n/n$ a.a.s.~the
  perturbed graph $G \cup G(n,p)$ contains $m$ pairwise vertex-disjoint copies
  of $C_\ell$.
\end{theorem}

For this result we are not aware of a construction that justifies the $\log n$-term and it
would be interesting to know if it can be omitted.
Also, we did not optimise the upper bound on $m$ stated in Theorem~\ref{thm:sublinear}, as Theorem~\ref{thm:non-extremal} and~\ref{thm:extremal}  cover anyway the cases of larger values of $m$ and $p \ge C \log n/n$.

We remark that, in the randomly perturbed graph setting, several variations of the problem we investigated can be considered.
For example, given an integer $\ell \ge 3$, $0 < \delta \le 1/\ell$ and $\alpha \in (0,1)$, one can ask for the threshold for the property
that the randomly perturbed graph $G_\alpha \cup G(n,p)$ contains $\delta n$
pairwise vertex-disjoint copies of $C_\ell$, for any $n$-vertex graph $G_\alpha$ with minimum degree at least $\alpha n$.
The case $\delta=1/\ell$ corresponds to cycle-factors and has been the core of our work, so we now focus on the case $0 < \delta < 1/\ell$.
It is an easy corollary of our Theorem~\ref{thm:non-extremal} that, given an integer $\ell \ge 3$, $0 < \delta < 1/\ell$ , $\eps>0$, 
and any $n$-vertex graph $G$ with minimum degree $\delta(G) \ge \delta n$, we can a.a.s.~find $(\delta - \eps)n$ pairwise vertex-disjoint copies of $C_\ell$
in $G \cup G(n,p)$ provided $p \ge C/n$, where $C$ is a large enough constant 
depending only on $\ell$  and $\eps$.
In other words we get the following.
\begin{corollary}
	\label{thm:generel}
	Given any integer $\ell \ge 3$, $\eps>0$, and $0<\delta<1/\ell$, there exists $C>0$ such that for any
	$n$-vertex graph $G$ with minimum degree $\delta(G) \ge (\delta+\eps)n$ we can a.a.s.~find $\delta n$ pairwise disjoint copies of $C_\ell$ in $G \cup G(n,p)$, provided that $p\ge C /n$.
\end{corollary}
This gives a lower bound on the threshold for any $\alpha > \delta$; moreover it is optimal for $\alpha < \tfrac{\ell \delta}{2}$ when $\ell$ is even and $\alpha<\tfrac{1+\delta}{2}$ when $\ell$ is odd (see the discussion after Theorem~\ref{thm:main} for the explanations of these bounds).
When $\alpha=\delta$, the threshold is $\log n/n$ as discussed in the first part of Theorem~\ref{thm:zero}.
When $0 \le \alpha < \delta$, the deterministic graph does not help and the threshold in $G(n,p)$ was determined by Ruci\'{n}ski~\cite{rucinski1992matching}.

Theorem~\ref{thm:main} follows from
Theorem~\ref{thm:non-extremal},~\ref{thm:extremal}, and~\ref{thm:sublinear}. The
proofs of Theorem~\ref{thm:non-extremal} and~\ref{thm:extremal} closely follow
the corresponding proof for a triangle-factor in~\cite{triangle_paper}, once all
lemmas are adjusted to the cycle setting. Therefore, we will only sketch their
proofs in Section~\ref{sec:overview}, together with a precise statement of each
lemma, and we refer the reader to~\cite{triangle_paper} for more details. It is
not hard to derive the new lemmas from the corresponding ones
in~\cite{triangle_paper}, and we skip their proof. Theorem~\ref{thm:sublinear}
requires new ideas, thus we will give a full proof in
Section~\ref{sec:sublinear} and Section~\ref{sec:appendix}.

\section{Sketch of the proofs of the Stability and Extremal Theorems}
\label{sec:overview}

For simplicity we assume $\alpha=1/\ell$ and that $G$ is an $n$-vertex graph with minimum degree $\delta(G) \ge n/\ell$, with $n$ being a multiple of $\ell$, in which case both Theorems~\ref{thm:non-extremal} and~\ref{thm:extremal}, give a $C_\ell$-factor in $G \cup G(n,p)$.
The proofs for smaller $\alpha$ follow along the same lines.
As $\ell$ is fixed, throughout all the section when we say \emph{cycle}, this always refers to a cycle of length $\ell$.
We will use the Szemer\'edi regularity lemma for Theorem~\ref{thm:non-extremal} and the concepts of regular and super-regular pairs in both proofs.
We use the degree form of the regularity lemma in~\cite{simon96}, and more details can be found there.

\subsection{Extremal Theorem.} 

Let $0<\beta \ll \eps \ll d \ll 1$ and $C>0$ be such that the following holds. 
Let $G$ be a $(1/\ell,\beta)$-stable graph on $n$ vertices with $\ell | n$ and $\delta(G) \ge n/\ell$.
To cover vertices with cycles we will repeatedly use that in any set of size $\beta n$ there is a path on $\ell-1$ vertices with edges of $G(n,p)$.
This holds a.a.s.~with $p \ge C \log n/n$ using a standard application of Janson's inequality.

As $G$ is $(1/\ell,\beta)$-stable, there exists a partition of $V(G)$ into $A \cup B$ where $|A|=(1/\ell \pm \beta) n$ and $|B|=(1-1/\ell \pm \beta) n$ such that all conditions in Definition~\ref{def:stability} are satisfied.
We can find a cycle factor in $G \cup G(n,p)$ in three steps.
Firstly, we find a collection of disjoint cycles $\cF_1$, such that after removing the cycles of $\cF_1$, we are left with two sets $A_1:=A \setminus V(\cF_1)$ and $B_1:=B \setminus V(\cF_1)$ such that $|B_1| = (\ell -1) |A_1|$.
The way we find these cycles depends on the size of $A$ and $B$.
If $|B| \ge (\ell-1)n/\ell$, then $|B| =(\ell-1) n/\ell+m$ for some $0 \le m \le \beta n$ and we have to find $m$ disjoint cycles entirely within $B$, just using the minimum degree $\delta(G[B]) \ge n/\ell-|A|=m$ and random edges.
This can be done using our Theorem~\ref{thm:sublinear}, and we let $\cF_1$ be the family of disjoint cycles we get.
Otherwise $|B| < (\ell-1)n/\ell$ and $|A|=n/\ell+m$ for some $1 \le m \le \beta n$, and we let $\cF_1$ be any family of $m/(\ell-2)$ disjoint cycles in $G \cup G(n,p)$ each with $\ell-1$ vertices in $A$ and one vertex in $B$.
Such a family can be found greedily: indeed during the process, there is always a vertex $v$ in $B$, not yet contained in a cycle, with at least $d(v,A) - (\ell -1)m/(\ell-2) \ge \beta n$ uncovered neighbours in $A$ in the graph $G$ and thus there is a path on $\ell-1$ vertices in its neighbourhood in $G(n,p)$, that completes to a cycle in $G \cup G(n,p)$.
Notice that the minimum degree of $G[A_1,B_1]$ is still linear in $n$ and all but few vertices of $A_1$ and $B_1$ have high degree to the other part (in fact they see all but few vertices in the other part).

In the second step we want to cover those vertices in $A_1$ and $B_1$ that do not see all but $10\beta n$ vertices from the other side.
We will cover them (and some other vertices) with two collections of cycles $\cF_2$ and $\cF_3$ respectively, where each cycle has one vertex in $A_1$ and $(\ell-1)$ vertices in $B_1$ so that we still have $|B_2| = (\ell-1) |A_2|$, where $A_2:=A_1 \setminus V(\cF_2 \cup \cF_3)$ and $B_2:=B_1 \setminus V(\cF_2 \cup \cF_3)$. 
That can be done greedily as above, just using the minimum degree condition in $G[A_1,B_1]$ and random edges, because there are only few vertices that do not have high degree.
Notice that after this each vertex from $A_2$ and $B_2$ sees all but $10 \beta n$ vertices from the other side, because we covered all vertices of smaller degree, and these sets are still large, because we only removed few cycles.

Finally we split $B_2$ arbitrarily into $\ell-1$ subsets $B_2^i$ of equal size, for $i=1, \dots, \ell-1$, and we remark that $|A_2| =|B_2^1|=\dots=|B_2^{\ell-1}|$. 
It is straightforward to check that $(A_2,B_2^1)$ and $(A_2,B_2^{\ell-1})$ are $(\eps,d)$-super-regular pairs, as each vertex has large degree to the other part.
Using random edges between $B_2^i$ and $B_2^{i+1}$ for $0 <  i <  \ell-1$ and the following Lemma, we can cover $A_2 \cup B_2^1 \cup \dots \cup B_2^{\ell-1}$ with a cycle factor $\cF_4$.

\begin{lemma}
	\label{lem:tripartite}
	For any $0<d<1$ there exists an $\eps>0$ and a $C>0$ such that the following holds.
	Let $\ell \ge 3$ be an integer and $V,U_1,U_2,\dots,U_{\ell-1}$ be sets of size $n$ such that $(V,U_1)$ and $(V,U_{\ell-1})$ are $(\varepsilon,d)$-super regular pairs, and for each $0<i<\ell-1$ let $G(U_i,U_{i+1},p)$ be a random bipartite graph with $p\ge C \log n/n$ .
	Then a.a.s.~there exists a $C_\ell$-factor.
\end{lemma}

We conclude by observing that the collection of cycles $\cF_1 \cup \cF_2 \cup \cF_3 \cup \cF_4$ gives a $C_\ell$-factor in $G \cup G(n,p)$.

\subsection{Stability Theorem.} 

Let $0<\eps \ll \gamma \ll  d \ll \beta < 1/(4\ell)$ and $C>0$.
Let $G$ be a graph on $n$ vertices with $\ell |n$ and $\delta(G) \ge (1/\ell - \gamma) n$ that is not $(1/\ell,\beta)$-stable.
We apply the regularity lemma to $G$ and obtain the reduced graph $R$, whose vertices are the clusters and there is an edge between two clusters if they give an $(\eps,d)$-regular pair in $G$.
By adapting ideas from~\cite{BMS_cover}, we proved the following stability result in~\cite{triangle_paper}.

\begin{lemma}[Lemma~4.4 in~\cite{triangle_paper}]
	\label{lem:stable_cluster}
	For any $0<\beta<\tfrac{1}{12}$ there exists a $d>0$ such that the following holds for any $0<\eps<d/4$, $4 \beta \le \alpha \le \tfrac 13$, and $t \ge \tfrac{10}{d}$ .
	Let $G$ be an $n$ vertex graph with minimum degree $\delta(G) \ge (\alpha -\tfrac12 d)n$ that is not $(\alpha,\beta)$-stable and let $R$ be the $(\eps,d)$-reduced graph for some $(\eps,d)$-regular partition $V_0,\dots,V_t$ of $G$.
	Then $R$ contains a matching $M$ of size $(\alpha+2d)t$.
\end{lemma} 

Using Lemma~\ref{lem:stable_cluster} and that the reduced graph inherits a minimum degree condition from $G$, we can cover $V(R)$ with pairwise vertex-disjoint stars, each with at most $\ell-1$ leaves, such that there are not too many stars isomorphic to $K_{1,\ell-1}$.
For simplicity, we only want to work with copies of stars isomorphic to $K_{1,1}$ or $K_{1,\ell-1}$ and, for that, we appropriately split each star to get a cover of $V(R)$ with pairwise disjoint copies of $K_{1,1}$ and $K_{1,\ell-1}$, such that there are still not too many copies of $K_{1,\ell-1}$.
For the rest of the section, we call stars (resp. matching edges) the copies of $K_{1,\ell-1}$ (resp. $K_{1,1}$) in $R$.
We make each edge super-regular (both in stars and matching edges) by removing some vertices and adding them to $V_0$, while keeping all clusters of the same size.
Then we remove a few more vertices from all but the centre cluster of each star and add them to $V_0$, in order to make each centre cluster of a star bigger than the other clusters.
Finally, by moving a few more vertices to $V_0$, we can assume that for all stars and matching edges in $R$, the number of vertices in the clusters together is divisible by $\ell$ (and thus $|V_0|$ is divisible by $\ell$ as well).
We can do all this such that $V_0$ does not get too large.

We start by covering $V_0$ with a collection of cycles $\cF_1$.
For this we use that with $p\ge C /n$ we can a.a.s.~assume that we can find short paths in $G(n,p)$ with vertices in predefined sets that are not too small and that for any regular pair we can find a cycle using one of its edges and the other vertices within predefined sets that are not too small using $\ell-1$ edges from $G(n,p)$.
Although covering $V_0$ could be done greedily just using the minimum degree condition and random edges, we do this more carefully in such a way that the total number of vertices in the clusters of each super-regular star or matching edge remains a multiple of $\ell$ (to avoid divisibility issues later), and that none of the centre cluster of a star gets significantly smaller.
Notice this can be guaranteed because we cover $V(R)$ without using not too many stars and, therefore, every vertex from $V_0$ has high degree into the non-centre clusters.
By constructing another collection of cycles $\cF_2$, we modify just the stars to ensure the sizes of the clusters are as required by the following lemma.

\begin{lemma}
	\label{lem:tripartite2}
	For any $0<\delta'\le d<1$ there exist $\delta_0,\delta,\eps$ with $\delta'\ge \delta_0>\delta>\eps>0$ such that given an integer $\ell \ge 3$, there exists $C=C(\ell)$ such that the following holds.
	Let $V,U_1,U_2,\dots,U_{\ell-1}$ be sets of size  $|V|=n$ and $(1-\delta_0) n \le |U_1|=\dots=|U_{\ell-1}| \le (1-\delta)n$, where $|V|+|U_1|+\dots+|U_{\ell-1}| \equiv 0 \pmod \ell$.
	Further for each $1 \le i \le \ell-1$ assume that $(V,U_i)$ is $(\varepsilon,d)$-super regular pair and for each $0<i<\ell-1$ let $G(V,p)$ and $G(U_i,U_{i+1},p)$ be random graphs with $p\ge C/n$.
	Then a.a.s.~there exists a $C_\ell$-factor.
\end{lemma}

From Lemma~\ref{lem:tripartite2} we also derive the following result about the existence of a cycle factor in a super-regular edge, again with the help of the random edges.

\begin{lemma}
	\label{lem:bipartite}
	For any $0<d<1$ there exist $\eps>0$ and $C>0$ such that the following holds for any integer $\ell \ge 3$ and any sets $U,V$ of size $|V|=n$ and $3n/4 \le |U| \le n$, where $|V|+|U| \equiv 0 \pmod \ell$.
	If $(U,V)$ is an $(\varepsilon,d)$-super-regular pair and $G(U,p)$ and $G(V,p)$ are random graphs with $p\ge C /n$, then a.a.s.~there exists a $C_\ell$-factor.
\end{lemma}

We use Lemma~\ref{lem:tripartite2} on each star and Lemma~\ref{lem:bipartite} on each matching edge to cover the remaining vertices with a collection of cycles $\cF_3$.
Together $\cF_1 \cup \cF_2 \cup \cF_3$ gives a $C_\ell$-factor in $G \cup G(n,p)$.

\section{Proof of the Sublinear Theorem}
\label{sec:sublinear}

Let $\ell \ge 3$ be an integer.
Let $1 \le m \le  \frac{n}{64 \ell^2}$ and $G$ be an $n$-vertex graph with minimum degree $\delta(G) \ge m$.
We let $p \ge C \log n/n$, with $C$ large enough such that a.a.s.~the applications of the propositions we state later hold and we a.a.s.~have the following properties hold in $G(n,p)$:
\begin{enumerate}[label=\rom]
	\item for any set of vertices $U$ of size $n/(64 \ell)$ there is a path on $\ell-1$ vertices in $G(n,p)[U]$ and \label{random_graph:path} 
	\item for a given set of vertices $U$ of size $n/2$ there are at least $\log^\ell n$ pairwise vertex-disjoint $C_\ell$'s in $G(n,p)[U]$. \label{random_graph:cycle} 
\end{enumerate}

Notice that~\ref{random_graph:path} can be guaranteed using the Janson's inequality and the union bound, while~\ref{random_graph:cycle} follows from~\cite[Theorem 3.29]{JLR}.
We want to show that a.a.s.~there exist $m$ pairwise vertex-disjoint copies of $C_\ell$ in $G \cup G(n,p)$.

Any vertex $v$ of large degree in $G$ can easily be covered by a cycle using~\ref{random_graph:path}.
If there are enough of these vertices, we can already claim $m$ cycles, otherwise we first ignore these vertices and cover them later.
For this, let $V'$ be the set of vertices from $G$ of degree at least $\frac{n}{64 \lfloor \ell/2 \rfloor}$.
If $|V'|\ge m$, then we can greedily find $m$ disjoint cycles in $G \cup G(n,p)$, each containing exactly one vertex from $V'$.
Indeed, as long as we have less than $m$ cycles, there is a vertex $v \in V'$ not yet contained in a cycle.
Then there are at least $\frac{n}{64 \lfloor \ell/2 \rfloor}- \ell m \ge \frac{n}{32 \ell} - \frac{n}{64 \ell} \ge \frac{n}{64 \ell}$ vertices $U \subseteq N_G(v)$ not covered by cycles and we can find a path on $\ell-1$ vertices within $G(n,p)[U]$ using property~\ref{random_graph:path}, that gives us a cycle on $\ell$ vertices containing $v$.
Otherwise, $|V'|<m$ and we remove $V'$ from $G$ to obtain $G'=G[V \setminus V']$.
Note that we have $v(G') = n-|V'| \ge n/2$, minimum degree $\delta(G') \ge m-|V'|=m'$, and maximum degree $\Delta(G') < \frac{n}{64 \lfloor \ell/2 \rfloor} \le \frac{v(G')}{32 \lfloor \ell/2 \rfloor}$.

Now the split the proof in three ranges for the value of $m'$:
\[
m' < \log^\ell n, \quad \log^\ell n \le m' \le M \sqrt{v(G')}, \quad \text{and} \quad M \sqrt{v(G')} \le m' \le \frac{n}{64 \ell^2}
\]
where $M$ is the constant given by Propositon~\ref{prop:gesqrt} below with input $\ell$.
If $m' < \log^\ell n$, then we a.a.s.~find $m'$ cycles $C_\ell$'s in $G(n,p)$ using~\ref{random_graph:cycle}.
If $\log^\ell n \le m' \le M \sqrt{v(G')}$, then we also have $\log^\ell v(G') \le \log^\ell n \le m' \le M \sqrt{v(G')}$, so we apply the following proposition to $G'$ with  $\gamma=1/(32 \lfloor \ell/2 \rfloor)$, and a.a.s.~find at least $m'$ vertex-disjoint cycles in $G \cup G(n,p)[V \setminus V']$.

\begin{proposition}
	\label{prop:lesqrt}
	Let $\ell \ge 3$ be an integer.
	For any $M \ge 1$ and $0<\gamma<1/2$ there exists $C>0$ such that for any $\log^\ell n \le m \le M \sqrt{n}$ and any $n$-vertex graph $G$ with maximum degree $\Delta(G) \le \gamma n$ and minimum degree $\delta(G) \ge m$ the following holds.
	With $p \ge C \log n/n$ there are a.a.s.~at least $m$ disjoint $C_\ell$'s in $G \cup G(n,p)$.
\end{proposition}

Finally, if $M \sqrt{v(G')} \le m' \le \frac{n}{64 \ell^2}$ then we also have $M \sqrt{v(G')} \le m' \le \frac{n}{64 \ell^2} \le \frac{v(G')}{16 \lceil \ell/2 \rceil}$, and, given the choice of $M$, we can apply the following proposition to $G'$ and again a.a.s.~find at least $m'$ vertex-disjoint cycles in $G \cup G(n,p)[V \setminus V']$.

\begin{proposition}
	\label{prop:gesqrt}
	Let $\ell \ge 3$ be an integer.
	There exist $M=M(\ell) \ge 1$ and $C=C(\ell)>0$ such that for any $M \sqrt{n} \le m \le \frac{n}{16 \lceil \ell/2 \rceil}$ and any $n$-vertex graph $G$ with maximum degree $\Delta(G) <\frac{n}{32 \lfloor \ell/2 \rfloor}$ and minimum degree $\delta(G) \ge m$ the following holds.
	With $p \ge C \log n/n$ there are a.a.s.~at least $m$ disjoint $C_\ell$'s in $G \cup G(n,p)$.
\end{proposition}

Now, after we found $m'$ disjoint cycles, we can greedily add cycles by using the $m-m'$ vertices from $V'$ and a path in their neighbourhood until we have $m$ cycles.
Analogous to above, as long as we have less than $m$ cycles, each available vertex $v$ from $V'$ has at least $\frac{n}{64 \ell}$ neighbours not covered by cycles, and with~\ref{random_graph:path} we get a cycle in $G \cup G(n,p)$ containing $v$.
That completes the proof of Theorem~\ref{thm:sublinear}.

\smallskip

We now give an overview of the proofs of Propositions~\ref{prop:lesqrt} and~\ref{prop:gesqrt} and then a full proof in Section~\ref{sec:appendix}.
For Proposition~\ref{prop:lesqrt} we rely on the following lemma that we proved in our previous work~\cite[Lemma $7.3$]{triangle_paper} and that allows us to find many large enough pairwise vertex-disjoint stars in $G$.
Before stating it, we need to introduce some notation.
With $g_K \ge 2$ an integer, a \emph{star} $K$ is a graph on $g_K+1$ vertices with one vertex of degree $g_K$ (this vertex is called the \emph{centre}) and the other vertices of degree one (these vertices are called \emph{leaves}).

\begin{lemma}[Lemma~7.3 in~\cite{triangle_paper}]
	\label{lem:manystars}
	For every $0<\gamma<1/2$ and integer $s>0$ there exists an $\eps>0$ such that    
	for $n$ large enough and any $m$ with $2/\eps \le m \le \sqrt{n}$ the following holds.
	In every $n$-vertex graph $G$ with minimum degree $\delta(G) \ge m$ and
	maximum degree $\Delta(G) \le \gamma n$ there exists a family $\mathcal{K}$ 
	of vertex-disjoint stars in $G$ such that every $K \in {\mathcal K}$ has $g_K$ leaves with
	$\eps m \le g_K \le \eps \sqrt{n}$ and
	$$\sum_{K \in \cK} g_K^2 \ge s \eps^2 n m.$$
\end{lemma}

With $\cK$ being the family of stars given by Lemma~\ref{lem:manystars}, we show that a.a.s.~at least $m$ stars of $\cK$ can be completed to cycles using edges of $G(n,p)$.
When $\ell=3$, it suffices to find one random edge within the set of leaves of a star $K$, for at least $m$ different stars $K \in \cK$ (see the proof of Proposition~7.1 in~\cite{triangle_paper} for details).
However, when $\ell>3$, we cannot find enough paths of length $\ell-2$ within the sets of leaves and, instead, proceed differently.
For each star $K \in \cK$, we split the set of its leaves into two sets $A_{K,2}$ and $A_{K,\ell}$ each of size $g_K/2$.
Then we find pairwise disjoint sets $A_{K,3},\dots,A_{K,\ell-1}$, each of size $g_K/2$, such that for $i=3,\dots,\ell-1$, every vertex from $A_{K,i}$ has at least one neighbour in $A_{K,i-1}$ in the random graph $G(n,p)$.
This can be done using expansion properties of $G(n,p)$, and we can also guarantee that the sets are pairwise disjoint for every $K \in \cK$.
Finally, analogously to the $C_3$ case, a.a.s.~we find an edge between $A_{K,\ell}$ and $A_{K,\ell-1}$ for at least $m$ different stars $K$.
Using the property of the new sets, by working backwards from $A_{K,\ell-1}$ to $A_{K,2}$, and adding two more edges from the star $K$, we obtain at least $m$ pairwise vertex-disjoint copies of $C_\ell$.

On the other hand, for Proposition~\ref{prop:gesqrt} and $m > M \sqrt{n}$, we cannot hope to find many large enough disjoint stars, and we need a different approach.
When $\ell$ is even, the proof is easy and follows from upper bounds on the extremal number of $C_\ell$, which is the maximum number of edges in an $n$-vertex graph that does not contain a copy of $C_\ell$. 
We find at least $m$ cycles greedily in $G$, as any $C_\ell$-free graph contains at most $\tfrac 34 n^{3/2}$ edges if $\ell=4$~\cite{turan_c4}, and at most $O(n^{1+2/\ell})$ edges if $\ell>4$~\cite{BS_TuranEvenCycles}.

For odd $\ell$, we use that $m \ge M \sqrt{n}$ is large enough to find an edge within the neighbourhood of each vertex, already with probability $q=\tfrac{C \log n}{m^2}$.
When $\ell=3$, we let $s=\lceil \tfrac nm \rceil$ and $t=\lceil \tfrac{m^2}{2n} \rceil$ and we find $s$ cycles $C_3$ in each of $t$ rounds.
More precisely, in each round we find $s$ vertices $v_1,\dots,v_s$ and pairwise disjoint sets of neighbours $B_1,\dots,B_s$ each of size $\lceil m/16 \rceil$.
Then we simply reveal edges with probability $q$ and get an edge of $G(n,q)$ within each of $B_1,\dots,B_s$.
As $tq \le C \tfrac{\log n}{n}$ we can repeat this for $t$ rounds and find $ts \ge m$ pairwise vertex-disjoint cycles $C_3$ (see the proof of Proposition~7.2 in~\cite{triangle_paper} for details).

When $\ell > 3$, we still find $s$ cycles $C_\ell$ in $t$ rounds, but this time, in each round and for each $i=1,\dots,s$, we do the following.
Instead of the vertex $v_i$, we construct a path in $G$ on $\ell-2$ vertices $v_{i,2},v_{i,3}, \dots, v_{i,\ell-1}$ and two sets $B_{i,2}$ and $B_{i,\ell-1}$ of at least $\lceil m/32 \rceil$ neighbours of $v_{i,2}$ and $v_{i,\ell-1}$, respectively.
We find this path using the dependent random choice technique, which is a powerful tool that, for example, gives upper bounds on extremal numbers of bipartite graphs (see the survey \cite{dependentRC}).
After having done that for each $i=1,\dots,s$, we find an edge of $G(n,q)$ between $B_{i,2}$ and $B_{i,\ell-1}$, that gives a cycle $C_\ell$.
As in the case $\ell=3$, we can perform $t$ rounds and find $ts \ge m$ pairwise vertex-disjoint cycles $C_\ell$.

\section{Proof of Proposition~\ref{prop:lesqrt} and~\ref{prop:gesqrt}}
\label{sec:appendix}

For the proof of Proposition~\ref{prop:lesqrt} we will use that the random graph $G(n,p)$ is expanding in the following sense:
\begin{lemma}
	\label{lem:expander}
	Let $G$ be a graph drawn from $G(n,p)$ and let $A$ and $B$ be two disjoint subsets of $V(G)$ with $1 \le |A| \le \sqrt{n}$ and $|B| \ge n/2$.
	Then with probability at least $1-1/n^2$ the subset $A$ has at least $|A|$ neighbours in $B$, provided that $p \ge 27 \log n/n$ and $n$ is large enough.
\end{lemma}

\begin{proof}{Proof of Lemma~\ref{lem:expander}}
	Let $A$ and $B$ be two disjoint subsets of $V$ with $|A| \le \sqrt{n}$ and $|B| \ge n/2$ and let $p \ge 27 \log n/n$ and $n \ge 36$.
	Assume that $A$ has less than $|A|$ neighbours in $B$.
	Then there exists $B' \subset B$ of size $|B|-(|A|-1)$ with $e(A,B')=0$. 
	The expected number of edges between $A$ and $B'$ is $|A| |B'| p$, thus the probability that there is no edge between $A$ and $B'$ is at most  $\exp(-|A||B'|p/3) \le \exp(-3|A| \log n)$ where we used the Chernoff's inequality and that $|B'| \ge n/3$.  
	By a union bound over all $\binom{|B|}{|A|-1} \le \exp(|A| \log |B|)$ choices for $B'$, the probability that there exists $B' \subset B$ of size $|B|-(|A|-1)$ with $e(A,B')=\emptyset$ is at most $\exp(-3 |A|\log n) \exp(|A| \log |B|) = \exp( -2|A| \log n) \le 1/n^2$.
	Thus,  $A$ has at least $|A|$ neighbours in $B$ with probability at least $1-1/n^2$.
\end{proof}

With this lemma we can prove the first proposition.

\begin{proof}{Proof of Proposition~\ref{prop:lesqrt}}
	Let $\ell \ge 3$ be an integer, $M \ge 1$ and $0<\gamma<1/2$.
	Let $C>0$ and $n$ be sufficiently large for the following arguments.
	With $\log^3 n \le m \le M \sqrt{n}$, let $G$ be an $n$-vertex graph with maximum degree $\Delta(G) \le \gamma n$ and minimum degree $\delta(G) \ge m$.
	We first find many disjoint stars in $G$ and then complete at least $m$ of them to cycles $C_\ell$ with the help of $G(n,p)$.
	
	Let $m'=\min(m,\sqrt{n})$.
	We apply Lemma~\ref{lem:manystars} to $G$ with $\gamma$, $s=\lceil 8/M \rceil$ and $m'$ to get $\eps'$ and, as $n$ is large enough and $m' \ge 2/\eps'$, we get a family $\cK$ of vertex disjoint stars on $V(G)$ such that $\eps' m' \le g_K \le \eps' \sqrt{n}$ for $K \in \cK$ and $\sum_{K \in \cK} g_K^2 \ge s \eps'^2 n m'$.
	We can assume w.l.o.g.~that $\eps' < \min \{ \frac{1}{4(s+1)} , \frac{1}{2(s+1)(\ell-2)}\}$.
	
	Thus as $m' \ge m/M$ and by setting $\eps=\eps'/M$, we have that $\eps m \le g_K \le \eps M \sqrt{n}$ for $K \in \cK$ and $\sum_{K \in \cK} g_K^2 \ge s M \eps^2 n m$.
	By deleting at most one vertex in each star of $\cK$, we can assume each star has an even number of leaf vertices (so $g_K$ is even for each $K \in \cK$) and, by deleting some stars of $\cK$, we can additionally assume that $s M \eps^2 n m \le \sum_{K \in \cK} g_K^2 \le (s+1) M \eps^2 n m$.
	Let $V' \subset V$ be the set of vertices not contained in any star of $\cK$.
	Then as $|\cK| \le \sum_{K \in \cK} g_K^2 /(\eps m)^2$, with the Cauchy-Schwarz inequality, we get that $\sum_{K \in \cK} g_K \le (s+1) M \eps n$, from which we conclude that $|V'| \ge 3n/4$ (using $\eps < \frac{1}{4M(s+1)}$).	
	For each star $K \in \cK$, we split the subset of its leaf vertices in half to get two subsets $A_{K,2}$ and $A_{K,\ell}$ each of size $g_K/2 \ge 1$.
	
	\begin{claim}
		\label{claim:expander}
		A.a.s~the following holds.
		For each $K \in \cK$ we can find $A_{K,3}, \dots, A_{K,\ell-1} \subset V'$ each of size $g_K/2$ such the sets $A_{K,i}$ with $K \in \cK$ and $3 \le i \le \ell-1$ are pairwise disjoint and for each $K \in \cK$, $3 \le i \le \ell-1$ and $v \in A_{K,i}$, there is a $w \in A_{K,i-1}$ such that $wv$ is an edge of $G(n,p)$. 
	\end{claim}
	
	\begin{proof}{Proof of Claim~\ref{claim:expander}}
		We construct such sets $A_{K,i}$ star after star.
		Let $K \in \cK$ and $W$ be the set of vertices not contained in any star and not yet used.
		Notice that at the beginning of the process $W=V'$ and that throughout the process $|W| \ge |V'| - \sum_{K \in \cK} (\ell-3) g_K/2 \ge 3n/4 - (\ell-3) (s+1) M \eps n/2 \ge n/2$, where the last inequality holds as $\eps < \frac{1}{2M(s+1)(\ell-2)}$.
		We will construct the sets $A_{K,3}, \dots, A_{K,\ell-1}$ iteratively. 
		Let $2 \le i \le \ell-1$ and assume $A_{K,3}, \dots, A_{K,i-1}$ have been constructed.
		We reveal random edges between $A_{K,i-1}$ and $W$.
		From Lemma~\ref{lem:expander} applied to $A=A_{K,i-1}$ and $B=W$, it follows that with probability at least $1-1/n^2$ the size of
		\begin{equation}
		\label{eq:neigh}
		\{ w \in W: \text{ there exists } v \in A_{K,i-1} \text{ with $vw$ being an edge of } G(n,p) \}
		\end{equation}
		is at least $|A_{K,i-1}| = g_K/2$ and thus we can choose $A_{K,i}$ to be any subset of~\eqref{eq:neigh} of size $g_K/2$.
		We then remove the $g_K/2$ vertices of $A_{K,i}$ from $W$.
		
		As there are at most $(s+1) M \eps n$ stars and for each star Lemma~\ref{lem:expander} needs to be applied $\ell-3$ times, by the union bound the probability that we succeed in all such applications is at least $1-(\ell -3)(s+1) M \eps n /n^2$.
		Thus a.a.s.~our process succeeds.
		It is clear that such collection of sets satisfies the claim.
	\end{proof}
	
	Notice that each random edge has been revealed at most once, and we have not revealed yet the random edges between $A_{K,\ell-1}$ and $A_{K,\ell}$.
	We now prove that a.a.s~there is an edge of $G(n,p)$ between $A_{K,\ell-1}$ and $A_{K,\ell}$ for at least $m$ distinct stars $K \in \cK$.

	As we have stars of different sizes, we split $\cK$ into $t=\lceil \log (M\sqrt{n}/m)/\log 2 \rceil$ subfamilies $\cK_i = \{ K \in \cK \,\colon 2^{i-1} \eps m \le g_K < 2^i \eps m \}$ for $1\le i \le t$,
	and set $k_i =|\cK_i|$. 	
	By deleting leaves, we may assume that all stars in $\cK_i$ have exactly $\lceil 2^{i-1} \eps m \rceil$ leaves.
	Denote by $\cI$ the set of indices $i \in [t]$ such that $k_i \left( 2^{i-1} \eps m \right)^2   \ge \eps^2 n m /t$.
	We prove that 
	$\sum_{i \in \cI} k_i \left( 2^{i-1} \eps m \right)^2 \ge \eps^2 nm.$
		
	Observe first that $\sum_{i \not\in \cI} k_i \left( 2^{i-1} \eps m \right)^2 \le t (\eps^2 nm/t)=\eps^2 nm$.

	It follows that
	\begin{eqnarray*}
		\sum_{i \in \cI} k_i \left( 2^{i-1} \eps m \right)^2 &=& \frac{1}{4} \sum_{i \in \cI} k_i \left( 2^i \eps m \right)^2 =
		\frac{1}{4} \sum_{i=1}^t k_i \left( 2^i \eps m \right)^2 
		-\sum_{i \not\in \cI} k_i \left( 2^{i-1} \eps m \right)^2\\
		&\ge&\frac{1}{4} \sum_{i=1}^t \sum_{K\in\cK_i} g_K^2 - \eps^2 nm\ge \frac{1}{4} s M \eps^2 nm - \eps^2 nm \ge \eps^2 nm.
	\end{eqnarray*}
	by the choice of $s$.
		
	\begin{claim}
		\label{claim:StarsTriangles}
		Let $i \in \cI$ and reveal the edges of $G(n,p)$ between $A_{K,\ell-1}$ and $A_{K,\ell}$ for each $K \in \cK_i$.
		Then with probability at least $1-1/n$, there is an edge of $G(n,p)$ between $A_{K,\ell-1}$ and $A_{K,\ell}$ for at least $k_i (2^{i-1}m)^2/ n$ distinct stars $K \in \cK_i$.
	\end{claim}
	
	Having this claim and since $|\cI| \le t =o (n)$, with a union bound over $i \in \cI$, there are a.a.s.~at least
	$\sum_{i \in \cI} k_i (2^{i-1}m)^2/ n = \frac{1}{\eps ^2 n} \sum_{i \in \cI} k_i \left( 2^{i-1} \eps m \right)^2 \ge \frac{\eps^2 nm}{\eps^2 n} \ge m$
	distinct stars $K \in \cK$ such that there is an edge between $A_{K,\ell-1}$ and $A_{K,\ell}$.
		
	For each such star $K$, we get a cycle $C_\ell$ in $G \cup G(n,p)$ in the following way: we start from the random edge between $A_{K,\ell-1}$ and $A_{K,\ell}$ given by Claim~\ref{claim:StarsTriangles}, work backwards using Claim~\ref{claim:expander} until $A_{K,2}$, and finally use the centre vertex of the star $K$ and close the cycle.
	Thus we get at least $m$ cycles and they are pairwise vertex-disjoint, as wanted.
	\end{proof}
	
	It remains to prove Claim~\ref{claim:StarsTriangles}.
	
	\begin{proof}{Proof of Claim~\ref{claim:StarsTriangles}}
		Fix $i \in \cI$ and let $k=k_i$ and $g = \lceil 2^{i-1} \eps m \rceil$.
		For $1 \le j \le k$, let $X_j$ be the indicator variable of the event that for the $j$-th star $K \in \cK_i$, there is an edge of $G(n,p)$ between $A_{K,\ell-1}$ and $A_{K,\ell}$, and set $X = \sum_{j=1}^k X_i$.
		Then $\PP[X_j=1]=1-(1-p)^{g^2/4}$ and $\EE[X] = k \left( 1-(1-p)^{g^2/4} \right)$.
		We have that $\EE[X] \ge 2 k g^2/n$.
		Indeed $	k \left( 1-(1-p)^{g^2/4} \right)  \ge 2 k g^2/n$ is equivalent to $1-2g^2/n \ge \left( 1- \frac{C \log n}{n}\right)^{g^2/4}$,
		and the later holds for large enough $C$ and $n$ using the inequality $1-x\le e^{-x}\le 1-\frac{x}{2}$ valid for $x<3/2$.
		
		From Chernoff's inequality and from the fact that $k g^2/n \ge \eps^2 m/t$ by the definition of $\cI$, it follows that with probability at most
		$$\exp \left( -\frac{1}{4} \frac{k g^2}{n} \right) \le \exp\left( - \frac{1}{4} \frac{\eps^2 m}{t} \right) \le \frac{1}{n}$$
		there are less than $kg^2/n$ stars $K \in \cK_i$ for which there are no edges of $G(n,p)$ between $A_{K,\ell-1}$ and $A_{K,\ell}$, where the last inequality holds as $t \le \log n$ and $m \ge \log^\ell n$.
	\end{proof}	 

We now turn to the proof of Proposition~\ref{prop:gesqrt}.

\begin{proof}{Proof of Proposition~\ref{prop:gesqrt}}
	Let $\ell \ge 3$ be an integer and $M=M(\ell)=16 \ell$.
	Let $C>0$ and $n$ be sufficiently large for the following arguments.
	Let $M \sqrt{n} \le m \le  \frac{n}{16 \lceil \ell/2 \rceil}$ and $G$ be an $n$-vertex graph $G$ with maximum degree $\Delta(G) <\frac{n}{32 \lfloor \ell/2 \rfloor}$ and minimum degree $\delta(G) \ge m$.
	We distinguish according to the parity of $\ell$.
	
	When $\ell$ is even, the proposition follows from the bound on the Turan's number of even cycles proved by Bondy and and Simonovits in \cite{BS_TuranEvenCycles}: they showed that for each even $\ell \ge 4$ there exists a constant $K=K(\ell)$ such that if $n$ is large enough and $H$ is a graph on $n$ vertices with more than $K n^{1+\frac{1}{\ell/2}}$ edges, then $H$ contains a $C_\ell$.
	Moreover when $\ell=4$, it follows from \cite{turan_c4} that $K(4)=3/4$ suffices.
	We find at least $m$ cycles $C_\ell$'s directly in $G$ without using any random edge, by repeatedly applying the cited results. 
	Indeed suppose we have been able to find a collection $\cC$ of $i<m$ disjoint $C_\ell$'s and let $G'=G[V \setminus V(\cC)]$.
	Then using that $|V(\cC)|=i \ell < m \ell$, we have $e(G') \ge \frac{n m}{2} - i \ell \Delta(G) \ge \frac{n m}{3} \ge \frac{Mn^{1+\frac{1}{2}}}{3} \ge K  v(G')^{1+\frac{1}{\ell/2}}$, thus $G'$ contains a cycle $C_\ell$.
	Notice that the last inequality is easily true for $\ell > 4$ and it is true when $\ell=4$ as $M = 64$ and $K \le 3/4$.
	
	Now we turn to odd cycles, and we closely follow our approach for triangles, with the differences outlined in Section~\ref{sec:sublinear}.
	We can greedily obtain a spanning bipartite subgraph $G' \subseteq G$ of minimum degree $\delta(G') \ge m/2$ by taking a partition of $V(G)$ into sets $A$ and $B$ such that $e_G(A,B)$ is maximised and letting $G'=G[A,B]$.
	Indeed, a vertex of degree less than $m/2$ can be moved to the other class to increase $e_G(A,B)$.
	W.l.o.g.~we assume $|B|\ge n/2\ge |A|$.
	Moreover, we have $|A| \ge 8 m \lfloor \ell/2 \rfloor$, as otherwise, with $e(A,B) \ge nm/4$, there is a vertex of degree at least $\frac{n}{32 \lfloor \ell/2 \rfloor}$, a contradiction.
	
	\begin{claim}	\label{claim:A'B'}
		For every $A' \subseteq A$, $B' \subseteq B$ with $|A'| \le 2m \lfloor \ell/2 \rfloor$ and $|B'| \ge n/4$, we have $e(A \setminus A',B') \ge n m/16$.
	\end{claim}
	\begin{proof}{Proof of Claim~\ref{claim:A'B'}}
		If $e(A \setminus A',B') < n m/16$, it follows from $e(A,B') \ge |B'| m /2\ge nm/8$ that we have $e(A',B') \ge n m /16$. Since $|A'| \le 2m \lfloor \ell/2 \rfloor$, there must be a vertex of degree at least $\frac{n}{32 \lfloor \ell/2 \rfloor}$ in $A'$, a contradiction.
	\end{proof}
	
	From this claim it follows that there are many vertices of high degree in $A \setminus A'$ and some of them have pairwise common neighbours in $B'$.
	\begin{claim}\label{claim:largeA*}
		Suppose that $A' \subseteq A$, $B' \subseteq B$ with $|A'| \le 2m \lfloor \ell/2 \rfloor$, $|B'| \ge n/4$. 
		Let $ A^* = \{ v \in A \setminus A' \,\colon\, \deg (v,B') \ge m/16 \}$. Then $|A^*| \ge \lfloor \ell/2 \rfloor m.$
		Moreover, there exist $\lfloor \ell/2 \rfloor$ distinct vertices $v_2, v_4, \dots, v_{\ell-1} \in A^*$ such that $v_i$ and $v_{i+2}$ have at least $\lfloor \ell/2 \rfloor -1$ common neighbours in $B'$ for each $i=2,4,\dots,\ell-3$.
	\end{claim}
	
	Before the proof, we recall some basic and common notation.
	Given $X, Y \subseteq V(G)$, with $N(X,Y)$
	we denote the common neighborhood of $X$ in $Y$, i.e. the set of all
	vertices of $Y$ that are adjacent to every vertex in $X$. Moreover when
	$v$ is a single vertex, $N(v,Y)=N(\{v\},Y)$.
	
	\begin{proof}{Proof of Claim~\ref{claim:largeA*}}
		We have $|A^*|\frac{n}{32 \lfloor \ell/2 \rfloor} + |A| m/16 \ge e(A^*,B') +e(A \setminus (A'\cup A^*),B') = e(A \setminus A',B') \ge \frac{nm}{16}$, where the last inequality uses Claim~\ref{claim:A'B'}.
		Since $|A| \le n/2$, we get
		$ |A^*| \ge \frac{nm/16-nm/32}{n/(32 \lfloor \ell/2 \rfloor)} =\lfloor \ell/2 \rfloor m.$
		Now we prove the second part of the claim using the dependent random choice technique (see the survey \cite{dependentRC}).
		Select any subset of $A^*$ of size $m/(8 \ell)$ and, by abusing the notation, call it $A^*$ again.
		Take a vertex $v \in B'$ uniformly at random and let $X=|N(v,A^*)|$.
		Then $\EE[X] = \sum_{w \in A^*} \PP[w \in N(v,A^*)] = \sum_{w \in A^*} \frac{N(w,B')}{|B'|} \ge \frac{|A^*|}{|B'|} \frac{m}{16}$.
		Let Y denote the random variable counting the number of subsets of $N(v,A^*)$ of size $2$ with fewer than $\ell/2$ common neighbours in $B'$.
		For a given subset $S$ of $A^*$ of size $2$, the probability that $S$ is a subset of $N(v,A^*)$ is $\frac{|N(S,B')|}{|B'|}$.
		As there are at most $\binom{|A^*|}{2}$ choice of $S$ for which $|N(S,B')| < \ell/2$, we have
		$\EE[Y] < \binom{|A^*|}{2} \frac{\ell/2}{|B'|}$.
		In particular $\EE[X-Y] \ge \frac{|A^*|}{4 |B'|}(\frac{m}{4}-|A^*| \ell) \ge \frac{m^2}{64 \cdot 4 \ell |B'|} \ge \ell/2$,
		where we used that $|A^*|=m/(8 \ell)$, $|B'| \le n$ and $m \ge 16 \ell \sqrt{n}$.
		Therefore there is a choice of $v$ for which $X-Y \ge \ell/2$.
		Consider such $v$ and delete one vertex from each subset $S$ of $N(v,A^*)$  of size $2$ with fewer than $\ell/2$ common neighbors. 
		Let $U$ be the remaining subset of $A^*$. 
		The set $U$ has at least $X - Y \ge \ell/2$ vertices and all its subsets of size $2$ have at least $\ell/2$ common neighbours.
		This proves the claim.
	\end{proof}
	
	In the given hypotheses, Claim~\ref{claim:largeA*} gives $\lfloor \ell/2 \rfloor$ vertices $v_2, v_4, \dots, v_{\ell-1} \in A^*$ that can be completed to a path $v_2v_3v_4\dots v_{\ell-1}$ by greedily choosing $\lfloor \ell/2 \rfloor -1$ vertices $v_3, v_5, \dots, v_{\ell-2}$ in $B'$. 
	Let $s=\lceil 2n/m \rceil$ and $t= \lceil \frac{m^2}{2n} \rceil$.
	We will now iteratively construct $m$ cycles $C_\ell$'s in $t$ rounds of $s$ cycles each.
	In each round we will reveal $G(n,q)$ with $q = \frac{C \log n}{m^{2}}$.
	For the start we set $A'=B_0=\emptyset$.
	
	Let $i=1,\dots,t$, suppose that before the $i$-th round we have $|A'| = (i-1)s \lfloor \ell/2 \rfloor < 2m \lfloor \ell/2 \rfloor$ and $|B_0|=(i-1)s \lceil \ell/2 \rceil < 2m \lceil \ell/2 \rceil$, and note this is true for $i=1$.
	For $j=1,\dots,s$ we pick a subset of $\ell-2$ distinct vertices $V_j=\{v_{j,2}, v_{j,3}, \dots, v_{j,\ell-1} \}$ with $v_{j,k} \in A \setminus A'$ for $k=2,4,\dots,\ell-1$ and $v_{j,k} \in B \setminus B_0$ for $k=3,5,\dots,\ell-2$, and two pairwise disjoint subsets $B_{j,2},B_{j,\ell-1} \subseteq B \setminus B_0$, disjoint also with $V_j$, such that the following is true: $v_{j,2}v_{j,3}\dots v_{j,\ell-1}$ is a path and $B_{j,2},B_{j,\ell-1} \subseteq B'$ are sets of $\lceil m/16 \rceil$ neighbours of $v_2$ and $v_{\ell-1}$, respectively.
	We can do that greedly for each $j=1,\dots,s$, by applying Claim~\ref{claim:largeA*} with $A' \cup_{t=1}^{j-1} (V_t \cap A)$ and $B \setminus \left( B_0 \cup_{t=1}^{j-1} \left( V_t \cup B_{t,2} \cup B_{t,\ell-1} \right) \right)$: note such application is possible as the size of the first set is at most $2m \lfloor \ell/2 \rfloor$ and the size of the second one is at least $n/4$ as $m \le \frac{n}{16 \lceil \ell/2 \rceil}$.
	
	Now we reveal random edges with probability $q$.
	By Chernoff's inequality  and the union bound we get that with probability at least $1-1/n^2$ we have at least an edge between $B_{j,2}$ and $B_{j,\ell-1}$ for each $j=1,\dots,s$.
	Such edges will complete each of the $s$ paths $v_{j,2} v_{j,3} \dots v_{j,\ell-1}$ to a cycle on $\ell$ vertices.
	We add the vertices we use from $A$ to $A'$ and those used from $B$ to $B_0$.
	Notice that $|A'|=is\lfloor \ell/2 \rfloor$ and $|B'|=i s \lceil \ell/2 \rceil$, as required at the beginning of the next round.
	We can repeat the above $t$ times because with $m \ge M \sqrt{n}$ we have $t q \le \frac{C \log n}{n} = p$.
	By a union bound over the $t=\lceil \frac{m^2}{2n} \rceil = o(n)$ rounds, we get that we succeed a.a.s.~and find $t s \ge m$ cycles $C_\ell$.
\end{proof}

\bibliographystyle{abbrv}
\bibliography{references}

\end{document}